\documentclass[pdflatex,sn-mathphys-num,a4paper]{sn-jnl}

\usepackage{graphicx}%
\usepackage{multirow}%
\usepackage{amsmath,amssymb,amsfonts}%
\usepackage{amsthm}%
\usepackage{mathrsfs}%
\usepackage[title]{appendix}%
\usepackage{xcolor}%
\usepackage{textcomp}%
\usepackage{manyfoot}%
\usepackage{booktabs}%
\usepackage{algorithm}%
\usepackage{algorithmicx}%
\usepackage{algpseudocode}%
\usepackage{listings}%

\theoremstyle{thmstyleone}%
\newtheorem{theorem}{Theorem}
\theoremstyle{thmstyletwo}%
\newtheorem{example}{Example}%
\newtheorem{remark}{Remark}%

\theoremstyle{thmstylethree}%

\newcommand{\R}{\mathbb{R}}
\newcommand{\xii}{{|\xi|}}
\newcommand{\lloc}{{\mathrm{loc}}}

\begin{document}

\title[Semilinear waves with damped oscillations]{The critical exponent for semilinear wave equations\\ with damped oscillations}


\author[1]{\fnm{Marcello} \sur{D'Abbicco}}\email{marcello.dabbicco@uniba.it}

\affil[1]{\orgdiv{Department of Mathematics}, \orgname{University of Bari}, \orgaddress{\street{Via E. Orabona 4}, \city{Bari}, \postcode{70125}, \country{Italy}}}


\abstract{We determine the existence exponent for the global small data solutions for a class of evolution equations with power nonlinearity under the effect of a dissipation that dampen the oscillations at low frequencies and produce overdamping at high frequencies. In the case of the wave equation with viscoelastic dissipation, the counterpart of blow-up in space dimension $n=3$ has been very recently proved by Wenhui Chen, showing that the existence exponent is optimal.}

\keywords{wave equation, critical exponent, dispersive estimates, damping}



\maketitle

\section{Introduction}\label{sec1}

In this paper, we study semilinear evolution equations with a dissipation
\begin{equation}\label{eq:CPgen} \begin{cases}
u_{tt} +Au + Bu_t =f(u), & t>0, \ x\in\R^n\\
u(0,x)=0,\\
u_t(0,x)=u_1(x).
\end{cases}\end{equation}
Here $A=-\Delta$ or a more general second order operator (see later, \eqref{eq:A}) and $B=(-\Delta)^{\frac\theta2}$ for $\theta\in(1,2]$ or a more general, possibly inhomogeneous, operator (see later, \eqref{eq:B} and \eqref{eq:b-low}-\eqref{eq:b-high}). We assume that the nonlinearity verifies the following inequality
\begin{equation}\label{eq:f}
f(0)=0,\qquad |f(u)-f(v)| \leq C\,|u-v|\,\big(|u|^{\alpha}+|v|^{\alpha}\big),\qquad u,v\in\R,
\end{equation}
with $\alpha>0$. In some cases, it is sufficient to assume~\eqref{eq:f} for $|u|,|v|\ll1$, that is,
\begin{equation}\label{eq:fsmall}
f(0)=0,\qquad |f(u)-f(v)| \leq C\,|u-v|\,\big(|u|^{\alpha}+|v|^{\alpha}\big),\qquad |u|,|v|\ll1.
\end{equation}
A case of interest for~\eqref{eq:f} is $f(u)=|u|^{1+\alpha}$,\footnote{Assumption~\eqref{eq:f} has the advantage that it only relies on Lipschitz regularity; several nonlinearities of classical interest, as $|u|^{1+\alpha}$, $\pm u|u|^{\alpha}$ verify assumption~\eqref{eq:f}. For instance, it is sufficient to apply the fundamental theorem of calculus to get
\[\begin{split}
||u|^\alpha-|v|^\alpha|
    & = \Big| \alpha\int_{|v|}^{|u|} s^{\alpha-1} \,ds\Big| \leq \alpha ||u|-|v||\,\sup \{ |u|^{\alpha-1},|v|^{\alpha-1}\} \\
    & \leq \alpha |u-v|\,(|v|^{\alpha-1}+|u|^{\alpha-1}).
\end{split}\]}
while for~\eqref{eq:fsmall} we can find more examples. In particular, if $f(u)=\sum_{i=j}^k f_j(u)$, with $f_j(u)$ verifying~\eqref{eq:f} for some power $\alpha_j$, then~\eqref{eq:fsmall} holds for $\alpha=\min_{j=1,\ldots,k}\alpha_j$.

We study the case in which the dissipation term $Bu_t$ dampens the oscillations at low frequencies, but it creates the overdamping at high frequencies, see later, \eqref{eq:DO}-\eqref{eq:OD}.

We look for determining (global-in-time) existence exponents for~\eqref{eq:CPgen}. By existence exponent~$q_c$ we mean that for $1+\alpha>q_c$ and for sufficiently small initial data in some space, there is a unique global-in-time solution (in general, a weak solution), possibly under some restrictions from above on $\alpha$.

The obtained exponent is of great interest when it is possible to prove that it is critical, in the sense that for $\alpha<q_c-1$, there are no global-in-time solutions, generally under some sign assumption on the initial data, but with no assumption on their size. The case $\alpha=q_c-1$ may belong to the existence or to the nonexistence region, according to different scenarios.

In absence of dissipation, i.e.. when $B=0$, the critical exponent for problem the semilinear wave equation~\eqref{eq:CPgen} with $A=-\Delta$ and $f(u)=|u|^{1+\alpha}$ was conjectured to be $q_c=\gamma(n-1)$ where $\gamma=\gamma(\kappa)$ is the solution to
\begin{equation}\label{eq:strauss}
\frac{\gamma-1}{\gamma+1}\,\frac\kappa2=\frac1\gamma.
\end{equation}
by W. Strauss~\cite[p.112]{Strauss} in space dimension~$n\geq2$:

\begin{quote}
\emph{``We conjecture that the critical power in John's theorem in $n$ space dimensions is $\gamma(n-1)$.''}
\end{quote}

The conjecture was motivated by the first results obtained by John in space dimension $n=3$ and by Glassey in space dimension $n=2$ and by a comparison for the critical exponents for other dispersive models. The result was later proved in a series of papers~\cite{A, Georgiev, GLS, G, G2, JZ, K5, LS96, Sc, Si, T, YZ06, Zhou}. The critical case $1+\alpha=\gamma(n-1)$ belongs to the nonexistence range.

The Strauss critical exponent~$\gamma(\kappa)$ in~\eqref{eq:strauss} is typical of dispersive models (see, for instance, \cite{Cazenave1992,NK02} for Schr\"odinger models, and~\cite{ChoOzawa2007DCDS,DAL26JMAA} for Boussinesq models), and this explains the reason behind the fact that $\kappa=n-1$ for the wave model, while $\kappa=n$ for the Schr\"odinger equation. The dispersion strength is indeed related to the rank of the Hessian of the phase function, and the Hessian of $\xii$ has rank $n-1$ (maximal value for a homogeneous function of degree~$1$), while the Hessian of $\xii^a$ has rank $n$ for $a\in(0,1)\cup(1,+\infty)$.

More explicitly, there is a strong connection behind the decay rate of dispersive and the exponent in~\eqref{eq:strauss}. In particular, for the wave model, a dispersive estimate of interest is
\[ \|u(t,\cdot)\|_{L^{2+\alpha}} \leq Ct^{-\frac{n-1}2\,\frac{\alpha}{\alpha+2}} \big( \|u_0\|_{\dot B_{1+\frac1{\alpha+1}},1}^s + \|u_1\|_{\dot B_{1+\frac1{\alpha+1}},1}^{s-1} \big),\quad s=\frac{n+1}2\,\frac{\alpha}{\alpha+2}, \]
and $\alpha+1$ times the exponent of $t$ is integrable for $t\in[1,+\infty)$ if, and only if,
\[ (\alpha+1)\,\frac{n-1}2\,\frac\alpha{\alpha+2} >1 \iff 1+\alpha>\gamma(n-1). \]
The situation becomes very different if a classic dissipation is added to the wave equation, that is, if $A=-\Delta$ and $Bu_t=u_t$ in~\eqref{eq:CPgen}. A. Matsumura~\cite{M76} determined the existence of small data global-in-time solutions for~$\alpha>2/n$ in space dimension~$n=1,2$. The result was extended to any space dimension~$n\geq3$ by G. Todorova and B. Yordanov~\cite{TY01}, with the nonexistence counterpart in the critical case proved in~\cite{Z01}. The critical case for more general nonlinearities has been recently discussed in~\cite{EGR} (see also~\cite{DAbbiccoGirardi2023,ChenReissig24,Girardi25,ChenGirardi25}). The exponent $1+2/n$ is related to the decay rate of the $L^1-L^{1+\alpha}$ damped wave equation, that is, the same decay rate of the heat equation; for the heat equation, the following estimate holds:
\[ \|u(t,\cdot)\|_{L^{1+\alpha}} \leq Ct^{-\frac{n}2\left(1-\frac1{\alpha+1}\right)} \|u(0,\cdot)\|_{L^1}. \]
In this case, the decay rate of $\|u(t,\cdot)\|_{L^{1+\alpha}}^{1+\alpha}$ is integrable for $t\in[1,+\infty)$ if, and only if,
\[ \frac{n}2\alpha >1 \iff \alpha>\frac2n. \]
Not surprisingly, the critical exponent $q_c=1+2/n$ is the same of the semilinear heat equation investigated by H. Fujita~\cite{F66}.

The analogy between the classical damped wave equation and the heat equation is due to the diffusion phenomenon showed for the associated linear model~\cite{HM, HT, MN03, N03}:
\begin{equation}\label{eq:CPlin} \begin{cases}
u_{tt} -\Delta u + Bu_t =0, & t>0, \ x\in\R^n\\
u(0,x)=u_0(x),\\
u_t(0,x)=u_1(x).
\end{cases}\end{equation}
At low frequencies, the action of the damping is so strong that it cancels the oscillations and make the solution's asymptotic profile the same of the heat equation. This scenario is physically interpreted as \emph{overdamping}.

To clarify the physical meaning, we may consider the equation for the damped harmonic oscillator
\[ \begin{cases}
y'' + by' + my =0, t>0,\\
y(0)=y_0,\\
y'(0)=y_1,
\end{cases} \]
with $b,m>0$. The solution exponentially decays as $t\to\infty$, but it has \emph{damped oscillations} if $b^2<4m$,
\begin{equation}\label{eq:DO}
y = e^{-\frac{b}2t}\left( y_0\,\cos(\omega t) + \frac{y_1+(b/2)y_0}{\omega}\,\sin(\omega t) \right),\qquad \omega=\sqrt{m-(b/2)^2},
\end{equation}
while for $b^2>4m$ the oscillations are canceled by the \emph{overdamping}:
\begin{equation}\label{eq:OD}
y = \frac{2y_1+(b+D)y_0}{2D}\,e^{-\frac{b-D}2t} - \frac{2y_1+(b-D)y_0}{2D}\,e^{-\frac{b+D}2t},\qquad D=\sqrt{b^2-4m}.
\end{equation}
The name ``overdamping'' comes from the fact that the following quantity is decreasing with respect to $b$:
\begin{equation}\label{eq:ODasymp}
\frac{b-D}2=\frac{2m}{b+D}\sim \frac{m}b,\quad \text{as $(b/m)\to\infty$.}
\end{equation}
In the limit case $b^2=4m$ there can be at most one oscillation: 
\[ y = \big(y_0+(y_1+(b/2)y_0)t\big)\,e^{-\frac{b}2t}. \]
When $B=(-\Delta)^{\frac\theta2}$ with $\theta\in(0,1)$ in~\eqref{eq:CPlin}, a diffusion phenomenon still holds, as discussed in~\cite{K00,DAE14JDE} and it leads to a critical exponent of Fujita type for~\eqref{eq:CPgen} that corresponds to the diffusive problem $\partial_t+(-\Delta)^{1-\frac\theta2}$, namely (see~\cite{DAR14,DAE14NA}):
\[ q_c = 1 + \frac2{n-\theta}. \]
In the special case $B=\mu(-\Delta)^{\frac12}$, the equation becomes scale-invariant and the oscillations are canceled in the solution to~\eqref{eq:CPlin} if, and only if, $\mu>2$. Nevertheless, for any $\mu>0$, the critical exponent for~\eqref{eq:CPgen} remains $q_c=1+2/(n-1)$ (see~\cite{DAR14,DA14proc}).

The situation was less clear for the case $B=(-\Delta)^{\frac\theta2}$ for $\theta\in(1,2]$. In this case, at low frequencies, one has damped oscillations for~\eqref{eq:CPlin}. However, the interplay between the dissipation related to the damping term, and the dispersion of the modified wave equation was not completely clear until recent times. In~\cite{S00}, Y. Shibata derived $L^1-L^1$ and $L^1-L^\infty$ estimates for the linear model~\eqref{eq:CPlin} with $B=-\Delta$. In particular, its result was optimal for $L^1-L^1$ estimates in odd space dimension, and for $L^1-L^\infty$ in any space dimension. The interpolation, however, did not produce sharp $L^1-L^q$ estimates for $q\in(1,\infty)$, due to the delicate nature of the interplay of dissipation and dispersion.

In~\cite{D'Abbicco-2025}, the effect of the interplay of dissipation and dispersion on $L^p-L^q$ estimates was clarified for the linear problem
\begin{equation}\label{eq:CPgenlin} \begin{cases}
u_{tt} +Au + Bu_t =0, & t>0, \ x\in\R^n\\
u(0,x)=0,\\
u_t(0,x)=u_1(x).
\end{cases}\end{equation}
The model considered included the study at low frequencies of~\eqref{eq:CPlin} when $B=(-\Delta)^{\frac\theta2}$, for $\theta\in(1,2]$.

In our paper, we apply the decay estimates obtained in~\cite[Theorem 6]{D'Abbicco-2025} for the solution to~\eqref{eq:CPgenlin}, to determine the corresponding existence exponent to the initial value problem~\eqref{eq:CPgen}.

During the preparation of this paper, the existence exponent $q_c$ in~\eqref{eq:qc} has been proved to be critical when $n=3$ and $A=B=-\Delta$ in~\cite{ChenPreprint}, in the sense that one has global-in-time small data solutions for $p>7/3$, and blow-up in finite time of the solutions when $p\leq 7/3$ (see later, Example~\ref{ex:3}).

This finding means that the $L^1-L^q$ estimates determined in~\cite{D'Abbicco-2025} for~\eqref{eq:CPgenlin} are, at least in some scenarios, an optimal tool to prove global-in-time existence of small data solutions to~\eqref{eq:CPgen} for supercritical powers. A more detailed discussion is provided in~\textsection\ref{sec:conclusion}.

We mention that when $A$ is homogeneous of degree $\sigma$ with $\sigma\neq1$, the critical exponent is very different in low space dimension (see~\cite{EL,DAE22}).


\section{Result}

In this paper, we assume that
\begin{equation}\label{eq:A}
A\varphi = \mathscr{F}^{-1}(\omega^2(\xi) \mathscr{F}\varphi), \qquad \varphi\in\mathcal S(\R^n),
\end{equation}
with $\omega$ homogeneous of degree~$1$, positive and smooth on the unit sphere (say, $\mathcal C^{n+3}(S^{n-1})$) and such that the Hessian matrix
\[ H_\omega = \begin{pmatrix}
\partial_{\xi_1}^2 \omega & \ldots & \partial_{\xi_1}\partial_{\xi_j}\omega & \ldots & \partial_{\xi_1}\partial_{\xi_n}\omega\\
\vdots & \vdots & \vdots & \vdots & \vdots \\
\partial_{\xi_1}\partial_{\xi_n}\omega  & \ldots & \partial_{\xi_j}\partial_{\xi_n}\omega & \ldots & \partial_{\xi_n}^2\omega
\end{pmatrix}  \]
has (maximal) rank $n-1$. Here and in the following, $\mathscr{F}$ denotes the Fourier transform with respect to the space variable (in general, for tempered distributions). We assume that
\begin{equation}\label{eq:B}
B\varphi = \mathscr{F}^{-1}(b(\xi) \mathscr{F}\varphi), \qquad \varphi\in\mathcal S(\R^n),
\end{equation}
with $b\in\mathcal C^{n+1}(\R^n\setminus\{0\})$ that verifies\footnote{The choice of the thresholds $\xii=2$ and $\xii=1/2$ is arbitrary, any couple of thresholds $2c$ and $c/2$ with $c>0$ would work; the conditions are equivalent, adjusting the constants $b_0,b_1$.}

\begin{align}
\label{eq:b-low}
b(\xi)\geq b_0\xii^{\theta_0}, \qquad \xii^{|\alpha|} |\partial_\xi^\alpha b(\xi)|\leq b_1\xii^{\theta_0}, \qquad |\alpha|\leq n+1, \quad \xii\leq 2\\
\label{eq:b-high}
b(\xi)\geq b_0\xii^{\theta_1}, \qquad \xii^{|\alpha|} |\partial_\xi^\alpha b(\xi)|\leq b_1\xii^{\theta_1}, \qquad |\alpha|\leq n+1, \quad \xii\geq \frac12,
\end{align}
with $\theta_0>1$ and $\theta_1\in(1,2]$.

The action of $A$ and $B$ is then extended by density to broader spaces than~$\mathcal S$.

In particular, conditions \eqref{eq:b-low}-\eqref{eq:b-high} hold for $\theta_1=\theta_2=\theta$ if $b$ is homogeneous of degree~$\theta\in(1,2]$, and $b\in\mathcal C^{n+1}(S^{n-1})$, positive.

To distinguish low frequencies and high frequencies estimates for the solution to~\eqref{eq:CPgenlin}, we fix $\chi\in\mathcal C_c^\infty(\R^n)$ such that $\chi=1$ for $\xii\leq1/2$ and $\chi=0$ for $\xii\geq2$, then we define the operators that project the solution to~\eqref{eq:CPgenlin}, at low and at high frequencies:
\[ P_0u=\mathscr{F}^{-1}(\chi\mathscr{F}u), \qquad P_1u=u-P_0u=\mathscr{F}^{-1}((1-\chi)\mathscr{F}u). \]
At low frequencies we rely on the following.
\begin{theorem}\label{thm:DAE}[Theorem 6 in~\cite{D'Abbicco-2025}]
Assume $\theta_0>1$. Fix $q\in(2,\infty]$ if $n=2$, $q\in(1,\infty]$ if $n=3$, or $q\in[1,\infty]$ if $n\geq4$. Then the solution to~\eqref{eq:CPgenlin} verifies the following low-frequencies $L^1-L^q$ estimate
\begin{equation}\label{eq:deltaq}\begin{split}
&\|P_0u(t,\cdot)\|_{L^q} \leq C\,(1+t)^{\delta_q}\,\|u_1\|_{L^1}, \qquad t\geq0,\\
&\delta_q=(n-1)\left(\frac1q-\frac12\right) - \frac1{\theta_0}\left(\frac{n-1}2 - \frac1q\right),
\end{split}\end{equation}
for some $C>0$.
\end{theorem}
We stress that $\delta_q$ is decreasing with respect to $q$.

%
At high frequencies, we prove the following.
\begin{theorem}\label{thm:high}
Assume $\theta_1\in(1,2]$ and fix $q\in(1,\infty]$ such that
\[
n\left(1-\frac1q\right)<\theta_1.
\]
Then the solution to~\eqref{eq:CPgenlin} verifies the following high-frequencies $L^1-L^q$ estimate
\[ \|P_1u(t,\cdot)\|_{L^q} \leq C\,e^{-ct}\,\|u_1\|_{L^1}, \qquad t\geq0, \]
for some $C,c>0$. Moreover, if $\theta_1\in(1,2)$ and
\[
\theta_1\leq n\left(1-\frac1q\right)<2,
\]
the solution to~\eqref{eq:CPgenlin} verifies the following high-frequencies $L^1-L^q$ estimate
\[ \|P_1u(t,\cdot)\|_{L^q} \leq C_\delta\,t^{\delta-1}\,e^{-ct}\,\|u_1\|_{L^1}, \qquad t\geq0, \]
for some $c,\delta>0$ and $C_\delta>0$ (with $C_\delta\nearrow\infty$ as $\delta\searrow0$).
\end{theorem}
We postpone the proof of Theorem~\ref{thm:high} to~\textsection\ref{sec:proofhigh}.

To find the existence exponent $q_c$, we impose that $\delta_{q_c}$ in~\eqref{eq:deltaq} verifies the equation $q_c\delta_{q_c}=-1$:\footnote{If $n=2$, we notice that
\[ q_c = 2+\frac{2\theta_0}{\theta_0+1}>2. \]}
\[ (n-1)\left(1-\frac{q_c}2\right) - \frac1{\theta_0}\left(q_c\frac{n-1}2 - 1\right) =-1  \iff q_c\frac{n-1}2\left(1+\frac1{\theta_0}\right) = n+\frac1{\theta_0}, \]
that is,
\begin{equation}\label{eq:qc}
q_c = \frac2{\theta_0+1}\,\frac{n\theta_0+1}{n-1}.
\end{equation}
First, we only assume small initial data in $L^1$.
\begin{theorem}\label{thm:main}
Let $\theta_0>1$ and $\theta_1\in(1,2]$, and assume~\eqref{eq:f} with $q_c-1<\alpha<\theta_1/(n-\theta_1)$ ($q_c-1<\alpha<\infty$ if $n=\theta_1=2$).\footnote{The interval $(q_c-1,\theta_1/(n-\theta_1))$ may be empty; in this case, Theorem~\ref{thm:main} is not applicable, though we may rely on Theorem~\ref{thm:inf}.} Fix $\bar q>1+\alpha$, and $\bar q<n/(n-\theta_1)$ if $n\geq3$. Then there exists $\varepsilon>0$ such that if
\[ u_1\in L^1, \quad \|u_1\|_{L^1}\leq \varepsilon, \]
there is a unique solution $u\in L^\infty_\lloc([0,\infty),L^1\cap L^{\bar q})$. Moreover, $u(t,\cdot)$ satisfies the same estimate as in~\eqref{eq:deltaq}, that is,
\[ \|u(t,\cdot)\|_{L^q}\leq C\,(1+t)^{\delta_q}\,\|u_1\|_{L^1},\]
for any:
\begin{itemize}
\item $q\in(2,\bar q]$ if $n=2$;
\item $q\in(1,\bar q]$ if $n=3$;
\item $q\in[1,\bar q]$ if $n\geq4$.
\end{itemize}
\end{theorem}
We postpone the proof of Theorem~\ref{thm:main} to~\textsection\ref{sec:proofmain}.

The restriction from above $q_c<1+\theta_1/(n-\theta_1)$ is related to the assumption that the initial data is only small in $L^1$, and it provides a technical bound from above on the space dimension~$n$, in view of the fact that $q_c\sim 2\theta_0/(\theta_0+1)>1$ as $n\to\infty$.
\begin{example}\label{ex:2}
When $\theta_0=2$, we may apply Theorem~\ref{thm:main} with $n=2$, $q_c=10/3$ and the estimate
%
\[ \|u(t,\cdot)\|_{L^q}\leq C\,(1+t)^{-\frac32\left(\frac12-\frac1q\right)}\,\|u_1\|_{L^1},\]
holds for any $q\in(2,\infty)$.
\end{example}
\begin{remark}\label{rem:-1}
As a consequence of $\frac1q>1-\frac{\theta_1}n$, we find
\[ \begin{split}
\delta_q & > (n-1)\left(1-\frac{\theta_1}n-\frac12\right) - \frac1{\theta_0}\left(\frac{n-1}2 - 1 + \frac{\theta_1}n\right) \\
    & = \left(\frac{n-3}2+\frac{\theta_1}n\right)\left(1-\frac1{\theta_0}\right)+1-\theta_1>-(\theta_1-1).
\end{split}\]
In the forthcoming Theorem~\ref{thm:inf}, it is of interest to check whether $\delta_q>-1$. We find
\[ \delta_q >-1 \iff \left(n-1+\frac1{\theta_0}\right)\frac1q>\frac{n-1}2\left(1+\frac1{\theta_0}\right)-1. \]
The above condition holds for any $q\leq\infty$ when $n=2$, and for
\[ q<q^\dagger(n,\theta_0) = \frac{n-1+\frac1{\theta_0}}{\frac{n-1}2\left(1+\frac1{\theta_0}\right)-1},\qquad n\geq3. \]
Clearly, $q_c<q^\dagger$, since $q_c$ is the value for which the $\delta_{q_c}=-1/q_c$.
\end{remark}
Following as in~\cite{DA14proc}, thanks to the overdamping at high frequencies, we may easily construct global-in-time solutions in $u\in L^\infty_\lloc([0,\infty),L^1\cap L^\infty)$, assuming small initial data in $L^1\cap L^p$, with $p>n/(2\theta_1)$. This allows us to remove the upper bound for $\alpha$ and, moreover, we may assume the weaker condition~\eqref{eq:fsmall} on the nonlinearity, in place of~\eqref{eq:f}.
\begin{theorem}\label{thm:inf}
Let $\theta_0>1$ and $\theta_1\in(1,2]$, and assume~\eqref{eq:fsmall} with $\alpha>q_c-1$. Then there exists $\varepsilon>0$ such that if
\[ u_1\in L^1\cap L^p, \quad \|u_1\|_{L^1}+\|u_1\|_{L^p}\leq \varepsilon, \quad \text{for some $p>n/(2\theta_1)$,} \]
there is a unique solution $u\in L^\infty_\lloc([0,\infty),L^1\cap L^\infty)$. Moreover, it satisfies the following estimate:
\[ \|u(t,\cdot)\|_{L^q}\leq C\,(1+t)^{\delta_q}\,\big(\|u_1\|_{L^1}+\|u_1\|_{L^\infty}\big),\]
for any:
\begin{itemize}
\item $q>2$ if $n=2$;
\item $q\in(1,q^\dagger)$ if $n=3$;
\item $q\in[1,q^\dagger)$ if $n\geq4$.
\end{itemize}
\end{theorem}
We postpone the proof of Theorem~\ref{thm:inf} to~\textsection\ref{sec:proofinf}.
\begin{remark}
Thanks to the fact that $u_1$ is small in $L^2$, by minor modifications, the solution space in Theorem~\ref{thm:inf} may be extended to include energy solutions, namely, we may find solutions in $L^\infty_\lloc([0,\infty),L^1\cap L^\infty)\cap \mathcal C([0,\infty),H^{\theta_1})\cap \mathcal C^1([0,\infty),L^2)$, also verifying a suitable energy estimate for
\[ \frac12\|u(t,\cdot)\|_{\dot H^{\theta_1}}^2+\frac12\|u(t,\cdot)\|_{L^2}^2. \]
\end{remark}
\begin{example}\label{ex:3}
When $\theta_0=2$, we may apply Theorem~\ref{thm:main} with $n=3$, $q_c=7/3$ and the estimate
%
\[ \|u(t,\cdot)\|_{L^q}\leq C\,(1+t)^{\frac5{2q}-\frac32}\,\big(\|u_1\|_{L^1}+\|u_1\|_{L^\infty}\big),\]
holds for any $q\in(1,\infty]$. For $A=B=-\Delta$, the blow-up in finite time for $q\leq 7/3$ has been recently proved by W. Chen~\cite{ChenPreprint} (the global existence of small data solution is also independently proved using the estimates in~\cite[Theorem 6]{D'Abbicco-2025}).
\end{example}
\begin{remark}\label{rem:Strauss}
After straightforward computation, we find that $q_c\leq \gamma(n-1)$ if, and only if,
\[ n\leq \bar n(\theta_0)=\frac{\theta_0^2+\theta_0+2}{\theta_0(\theta_0-1)} \]
or, equivalently,
\[ \theta_0\leq \bar\theta(n)= \frac{n+1 + \sqrt{n^2+10n-7}}{2(n-1)}. \]
We notice that the interval for $n$ is empty if
\[ \frac{\theta_0^2+\theta_0+2}{\theta_0(\theta_0-1)}<2 \iff \theta_0^2 -3\theta_0 -2>0 \iff \theta_0> \frac{3+\sqrt{17}}2, \]
while the bound $\bar n(\theta_0)\to\infty$ as $\theta_0\searrow1$.
\end{remark}


\section{Proof of Theorem~\ref{thm:high}}\label{sec:proofhigh}

Here we prove Theorem~\ref{thm:high}.
\begin{proof}[Proof of Theorem~\ref{thm:high}]
We may assume with no loss of generality that
\[ \mathscr{F}u(t,\xi)=K(t,\xi)\,\mathscr{F}u_1(\xi),\]
with
\begin{equation}\label{eq:Fourep}
K=\frac{e^{-\frac{b-D}2}-e^{-\frac{b+D}2}}{D}\,(1-\chi)\,,\qquad D=\sqrt{b(\xi)^2-4\omega^2(\xi)}>0,
\end{equation}
since this is true for sufficiently large $\xii$, due to~$\theta_1>0$ in~\eqref{eq:b-high} (see~\eqref{eq:OD}). By the asymptotics (see~\eqref{eq:ODasymp}):
\begin{equation}\label{eq:Fouasymp}
-\frac{b-D}2\sim -\frac{\omega^2}{b}, \qquad -\frac{b-D}2\sim -b,\qquad D\sim b,
\end{equation}
we obtain that~$\mathscr{F}^{-1}K(t,\cdot)\in L^q$, with
\begin{equation}\label{eq:ODest}
\|\mathscr{F}^{-1}K(t,\cdot)\|_{L^q}\leq C\,t^{-\frac\rho{2-\theta_1}}\,e^{-ct},
\end{equation}
for $\rho\geq0$, provided that
\[ n\left(1-\frac1q\right) < \theta_1 + \rho \]
In the case $\theta_1=1$, we put $\rho=0$. As $\rho\nearrow 2-\theta_1$, we find the desired condition.

Estimate~\eqref{eq:ODest} may be proved following as in~\cite{DAE14JDE,DAE14NA,DAE17NA}. 

However, we propose an alternative, shorter proof, based on the use of the real Hardy space $\mathscr H^1(\R^n)$ as a tool (see~\cite{FS72}).

Thanks to the Hardy Littlewood Sobolev inequality from $\mathscr H^1(\R^n)$ to $L^q(\R^n)$ (see~\cite[Theorem G]{SW}), we obtain
\[ \|\mathscr{F}^{-1}K(t,\cdot)\|_{L^q(\R^n)} \leq C\|\mathscr{F}^{-1}(\xii^{n\left(1-\frac1q\right)}\,K(t,\cdot))\|_{\mathscr H^1(\R^n)}.\]
For a given $\varepsilon>0$ such that
\[ n\left(1-\frac1q\right) + \varepsilon \leq \theta_1 + \rho , \]
with $\rho\geq0$ if $\theta_1\in(1,2)$ and $\rho=0$ if $\theta_1=2$, we now apply Mikhlin-H\"ormander theorem in the real Hardy space $\mathscr H^1(\R^n)$ (see~\cite[Theorem E]{Miyachi-1981}; see also~\cite[Theorems 4.6 and 4.7]{CALDERON1977101} and~\cite[Theorems 1 and 2]{Miyachi-1980}) to obtain that
\begin{equation}\label{eq:MH}
\|\mathscr{F}^{-1}(\xii^{n\left(1-\frac1q\right)}\,K(t,\cdot))\|_{\mathscr H^1(\R^n)} \leq C\,e^{-ct}\,t^{\delta-1}\,\|\mathscr{F}^{-1}((1-\chi)\xii^{-\varepsilon})\|_{\mathscr H^1(\R^n)},
\end{equation}
where the power $t^{-\frac{\rho}{2-\theta_1}}$ is omitted if $\theta_1=2$. To obtain~\eqref{eq:MH}, it is sufficient to notice that:
\[ \xii^{|\alpha|}\,|\partial_\xi^\alpha (\xii^{n\left(1-\frac1q\right)+\varepsilon} K)|\leq C\xii^{n\left(1-\frac1q\right)+\varepsilon-\theta_1}\, e^{-ct\xii^{2-\theta_1}}, \qquad |\alpha|\leq1+\frac{n}2, \]
for sufficiently large $\xii$ (namely, in the support of $1-\chi$). The above estimate follows from~\eqref{eq:Fourep}-\eqref{eq:Fouasymp} and from condition~\eqref{eq:b-high}.

If $\theta_1=2$ or if we want to take $\delta=1$ when $\theta_1\in(1,2)$, it is sufficient to notice that
\[ n\left(1-\frac1q\right)+\varepsilon-\theta_1\leq0 \]
for sufficiently small $\varepsilon>0$ when $n(1-1/q)\leq\theta_1$, and to use $e^{-ct\xii^{2-\theta_1}}\leq e^{-ct}$. When $\theta_1<2$, for any $\rho>0$, we may estimate
\[ e^{-ct\xii^{2-\theta_1}} \leq C\xii^{-\rho}\,t^{-\frac\rho{2-\theta_1}}\,\,e^{-ct}. \]
It is sufficient to notice that
\[ n\left(1-\frac1q\right)+\varepsilon-\rho-\theta_1\leq0 \]
holds for $n(1-1/q)<2$, letting $\rho\nearrow 2-\theta_1$ and $\varepsilon\searrow0$ (we set $\delta-1=-\frac\rho{2-\theta_1}$).

Therefore, we proved~\eqref{eq:MH}. The proof follows from the fact that $\mathscr{F}^{-1}((1-\chi)\xii^{-\varepsilon})\in\mathscr H^1(\R^n)$ (see, for instance, Corollary~1 in~\cite{DA26}).
\end{proof}


\section{Proof of Theorem~\ref{thm:main}}\label{sec:proofmain}

Here we prove Theorem~\ref{thm:main}.
\begin{proof}[Proof of Theorem~\ref{thm:main}]
We use a standard contraction argument for small data solutions.

Let $n\geq4$. For any~$T>0$, we introduce the solution space
\[ X(T)= L^\infty([0,T], L^1\cap L^{\bar q}), \]
equipped with norm
\begin{equation}\label{eq:norm}
\|u\|_{X(T)} = \sup_{t\in[0,T]} \Big( (1+t)^{-\delta_1}\|u(t,\cdot)\|_{L^1}+(1+t)^{-\delta_{\bar q}}\|u(t,\cdot)\|_{L^{\bar q}} \Big).
\end{equation}
In particular, by interpolation ($1+\alpha\in[1,\bar q]$), any function~$u$ in $X(T)$ verifies the decay estimate
\begin{equation}\label{eq:decay}
\|u(t,\cdot)\|_{L^{1+\alpha}} \leq (1+t)^{\delta_{1+\alpha}}\,\|u\|_{X(T)},\quad \forall t\in[0,T].
\end{equation}
A function $u\in X(T)$ is a solution to~\eqref{eq:CPgen} in~$X(T)$ if, and only if, it satisfies the equality
\begin{equation}\label{eq:fixedpoint}
u(t,\cdot) =  E(t)\ast u_1 + Fu(t,\cdot)\,, \qquad \text{in~$X(T)$,}
\end{equation}
where $E(t)\ast u_1$ denotes the solution to~\eqref{eq:CPgenlin}, and~$F$ is the operator such that, for any~$u\in X$,
\begin{equation}\label{eq:F0}
Fu(t,x) = \int_0^t E(t-s)\ast f(u(s,\cdot))\, ds\,.
\end{equation}
We claim that
\begin{align}
\label{eq:ubasic}
\|E(t)\ast u_1\|_{X(T)}
    & \leq C_1\,\|u_1\|_{\mathcal A},\\
\label{eq:contraction}
\|Fu-Fv\|_{X(T)}
    & \leq C_2\|u-v\|_{X} \bigl(\|u\|_{X(T)}^{\alpha}+\|v\|_{X(T)}^{\alpha}\bigr)\,,
\end{align}
with~$C_1$ and $C_2$ independent of~$T$. In particular, for $v=0$, in view of $f(0)=0$, estimate~\eqref{eq:contraction} reduces to
\begin{equation}\label{eq:well}
\|Fu\|_{X(T)} \leq C\|u\|_{X(T)}^{1+\alpha}\,.
\end{equation}
We now define
\[ R=2C_1\,\|u_1\|_{\mathcal A}. \]
For sufficiently small data, $2C_2R^{\alpha-1}\leq 1/2$. Then, by \eqref{eq:ubasic} and~\eqref{eq:well} it follows that the operator $E(t)\ast u_1 + F$ maps the ball $B_R=\{u: \ \|u\|_{X(T)}\leq R\}$ in itself. Due to~\eqref{eq:contraction}, it is a contraction. Therefore, there is a unique fixed point for $u^{\mathrm{lin}}(t,x) + F$ in $B_R$, that is, a unique solution to~\eqref{eq:fixedpoint}. Moreover, $\|u\|_{X(T)}\leq 2C_1\,\|u_1\|_{\mathcal A}$. Due to the fact that the constants do not depend on $T$, the result is global-in-time.

The proof of~\eqref{eq:ubasic} immediately follows by applying Theorems~\ref{thm:DAE} and~\ref{thm:high}, for $q=1,\bar q$. We shall now prove~\eqref{eq:contraction}.

By H\"older inequality, using~\eqref{eq:f} and~\eqref{eq:decay}, we find that
\begin{equation}\label{eq:Holder}\begin{split}
& \|(f(u)-f(v))(s,\cdot)\|_{L^1}\\
    & \qquad \leq C_1\,\|(u-v)(|u|^{\alpha}+|v|^{\alpha})(s,\cdot)\|_{L^1} \\
    & \qquad \leq C_1\,\|(u-v)(s,\cdot)\|_{L^{1+\alpha}}\,\big(\|u(s,\cdot)\|_{L^{1+\alpha}}^{\alpha}+\|v(s,\cdot)\|_{L^{1+\alpha}}^{\alpha}\big)\\
    & \qquad \leq C_2\,(1+s)^{(1+\alpha)\delta_{1+\alpha}}\,\|u-v\|_{X(T)}\,\big(\|u\|_{X(T)}^{\alpha}+\|v\|_{X(T)}^{\alpha}\big).
\end{split}\end{equation}
Therefore, applying Theorems~\ref{thm:DAE} and~\ref{thm:high}, by using~\eqref{eq:Holder}, for $q=1,\bar q$, we may estimate
\begin{equation}\label{eq:Fcontra}
\|(Fu-Fv)(t,\cdot)\|_{L^q}\leq C\,\big(I_0(t)+I_1(t)\big)\,\|u-v\|_{X(T)}\,\big(\|u\|_{X(T)}^{\alpha}+\|v\|_{X(T)}^{\alpha}\big),
\end{equation}
where
\begin{align}
\label{eq:I0}
I_0(t) &= \int_0^t (1+t-s)^{\delta_q}\,(1+s)^{(1+\alpha)\delta_{1+\alpha}}\,ds,\\
\label{eq:I1}
I_1(t) &= \int_0^t e^{-c(t-s)}\,(1+s)^{(1+\alpha)\delta_{1+\alpha}}\,ds.
\end{align}
Recalling that $\delta_q>-1$ (see Remark~\ref{rem:-1}), we get\footnote{Estimates of this type of integrals that originate from Duhamel's principle is classical; they go back at least to~\cite{Segal}.}
\[ I_0(t)\approx (1+t)^{\delta_q} \iff (1+\alpha)\delta_{1+\alpha}<-1 \iff 1+\alpha>q_c. \]
On the other hand,
\[\begin{split}
\int_0^{t/2} e^{-c(t-s)}\,(1+s)^{(1+\alpha)\delta_{1+\alpha}}\,ds&\leq C\, e^{-c_1t}\ll (1+t)^{\delta_q}, \\
\int_{t/2}^t e^{-c(t-s)}\,(1+s)^{(1+\alpha)\delta_{1+\alpha}}\,ds&\leq C\,(1+t)^{(1+\alpha)\delta_{1+\alpha}}\leq C\,(1+t)^{-1}\leq C\,(1+t)^{\delta_q}.
\end{split}\]
Therefore, we proved
\begin{equation}\label{eq:FuFv}
\|(Fu-Fv)(t,\cdot)\|_{L^q}\leq C\,(1+t)^{\delta_q}\,\|u-v\|_{X(T)}\,\big(\|u\|_{X(T)}^{\alpha}+\|v\|_{X(T)}^{\alpha}\big),
\end{equation}
that is, we obtained~\eqref{eq:contraction}, and this concludes the proof for $n\geq4$.

When $n=2,3$, we shall use a different $L^1-L^1$ low-frequencies estimate in place of~\eqref{eq:deltaq}, since the damping does not influence it, say
\[ \|P_0u(t,\cdot)\|_{L^1}\leq (1+t)^{\tilde\delta_1} \|u_1\|_{L^1} \, n=2,3. \]
The exact growth rate~$\tilde\delta_1$ is not important; it is proved that $\tilde\delta_1=1$ in~\cite{S00} for $A=B=-\Delta$ (moreover, $\tilde\delta_1=1$ also for $B=0$ if $n=2$ (see~\cite[Theorem 3]{D'Abbicco-2025}), and for $A=-\Delta$ if $n=3$).

Then we fix $q_0\in(2,1+\alpha]$ if $n=2$ or $q_0\in(1,1+\alpha]$ if $n=3$, and we replace~\eqref{eq:norm} by
\begin{equation}\label{eq:normlow}
\|u\|_{X(T)} = \sup_{t\in[0,T]} \Big( (1+t)^{-\tilde\delta_1}\|u(t,\cdot)\|_{L^1}+(1+t)^{-\delta_{q_0}}\|u(t,\cdot)\|_{L^{q_0}}+(1+t)^{-\delta_{\bar q}}\|u(t,\cdot)\|_{L^{\bar q}} \Big).
\end{equation}
By interpolation ($1+\alpha\in[q_0,\bar q]$), we still obtain~\eqref{eq:decay}, and the proof follows with minor modifications.
\end{proof}


\section{Proof of Theorem~\ref{thm:inf}}\label{sec:proofinf}

Here we prove Theorem~\ref{thm:inf}.
\begin{proof}[Proof of Theorem~\ref{thm:inf}]
We follow the proof of Theorem~\ref{thm:main}, assuming $n\geq4$ for brevity (the cases $n=2$ and $n=3$ are treated as in the proof of Theorem~\ref{thm:main}). However, now the solution space is
\[ X(T)= L^\infty([0,T], L^1\cap L^\infty), \]
equipped with norm
\begin{equation}\label{eq:norminf}
\|u\|_{X(T)} = \sup_{t\in[0,T]} \Big( (1+t)^{-\delta_1}\|u(t,\cdot)\|_{L^1}+(1+t)^{-\delta_{q^\dagger}}\big(\|u(t,\cdot)\|_{q^\dagger}+\|u(t,\cdot)\|_{L^\infty}\big) \Big).
\end{equation}
In particular, by interpolation, any function~$u$ in $X(T)$ verifies the decay estimate~\eqref{eq:decay} and
\begin{equation}\label{eq:decay2}
\|u(t,\cdot)\|_{L^{p(1+\alpha)}} \leq (1+t)^{\max\{\delta_{p(1+\alpha)},\delta_{q^\dagger}\}}\,\|u\|_{X(T)},\quad \forall t\in[0,T].
\end{equation}
The crucial difference with the proof of Theorem~\ref{thm:main} is related to the high frequencies estimate for $L^\infty$, since the $L^1-L^\infty$ does not hold in Theorem~\ref{thm:high}. By duality, the $L^1-L^q$ estimate in Theorem~\ref{thm:high} is the same of the $L^{q'}-L^\infty$ estimate with $q'=q/(q-1)$, H\"older conjugate of $q$. In particular, we may fix $q=p'=p/(p-1)$, with $p>n/(2\theta_1)$ as in the statement of Theorem~\ref{thm:inf}. That is, by the dual estimate of Theorem~\ref{thm:high}, we obtain
\begin{equation}\label{eq:dual}
\|P_1u(t,\cdot)\|_{L^\infty} \leq C\,e^{-ct}\,\|u_1\|_{L^p}.
\end{equation}
Trivially, we also get
\begin{equation}\label{eq:dualgen}
\begin{split}
\|P_1u(t,\cdot)\|_{L^q}
    & \leq \|P_1u(t,\cdot)\|_{L^1} +\|P_1u(t,\cdot)\|_{L^\infty} \\
    & \leq C\,e^{-ct}\,\big(\|u_1\|_{L^1}+\|u_1\|_{L^p}\big), \qquad q\in[1,\infty].
\end{split}
\end{equation}
The proof of~\eqref{eq:ubasic} immediately follows by applying Theorem~\ref{thm:DAE}, and~\eqref{eq:dualgen}, with $q=1,q^\dagger,\infty$ (recalling that $\delta_\infty\leq \delta_{q^\dagger}$). We shall now prove~\eqref{eq:contraction}.

Due to the fact that we will assume a contraction for $u\in X(T)$ with sufficiently small norm $\|u\|_{X(T)}$, and
\[ \|u(t,\cdot)\|_{L^\infty}\leq (1+t)^{\delta_{q^\dagger}}\,\|u\|_{X(T)}\leq \|u\|_{X(T)}, \qquad t\in[0,T], \]
it follows that we may assume $\|u\|_{L^\infty}$ sufficiently small. This allows us to use~\eqref{eq:fsmall} in place of~\eqref{eq:f}.

As in the proof of Theorem~\ref{thm:main}, we obtain~\eqref{eq:Holder}, using~\eqref{eq:fsmall} and~\eqref{eq:decay}. Similarly, using~\eqref{eq:fsmall} and~\eqref{eq:decay2}, we find
\begin{equation}\label{eq:Holder2}\begin{split}
& \|(f(u)-f(v))(s,\cdot)\|_{L^p}\\
    & \qquad \leq C_1\,\|(u-v)(|u|^{\alpha}+|v|^{\alpha})(s,\cdot)\|_{L^p} \\
    & \qquad \leq C_1\,\|(u-v)(s,\cdot)\|_{L^{p(1+\alpha)}}\,\big(\|u(s,\cdot)\|_{L^{p(1+\alpha)}}^{\alpha}+\|v(s,\cdot)\|_{L^{p(1+\alpha)}}^{\alpha}\big)\\
    & \qquad \leq C_2\,(1+s)^{(1+\alpha)\delta_{1+\alpha}}\,\|u-v\|_{X(T)}\,\big(\|u\|_{X(T)}^{\alpha}+\|v\|_{X(T)}^{\alpha}\big),
\end{split}\end{equation}
where in the last step we used that
\[ \max\{\delta_{p(1+\alpha)},\delta_{q^\dagger}\}\leq \delta_{1+\alpha}, \]
due to $1+\alpha\leq q^\dagger$.

We first apply Theorem~\ref{thm:DAE}, and~\eqref{eq:dualgen}, with $q=1,q^\dagger$, together with both~\eqref{eq:Holder} and~\eqref{eq:Holder2}, to obtain~\eqref{eq:Fcontra}, where $I_0(t)$ and $I_1(t)$ are as in~\eqref{eq:I0} and~\eqref{eq:I1}. Recalling that $\delta_{q^\dagger}>-1$, we obtain~\eqref{eq:FuFv}.

Now let $q=\infty$. In this case, applying Theorem~\ref{thm:DAE}, and~\eqref{eq:dualgen}, but using that $\delta_\infty\leq \delta_{q^\dagger}$, we obtain~\eqref{eq:Fcontra} with
\[ I_0(t) = \int_0^t (1+t-s)^{\delta_{q^\dagger}}\,(1+s)^{(1+\alpha)\delta_{1+\alpha}}\,ds,\]
and $I_1(t)$ as in~\eqref{eq:I1}. That is, we obtain
\[ \|(Fu-Fv)(t,\cdot)\|_{L^\infty}\leq C\,(1+t)^{\delta_{q^\dagger}}\,\|u-v\|_{X(T)}\,\big(\|u\|_{X(T)}^{\alpha}+\|v\|_{X(T)}^{\alpha}\big), \]
in place of~\eqref{eq:FuFv}. Still, this concludes the proof of~\eqref{eq:contraction}, in view of the definition of $\|u\|_{X(T)}$.
\end{proof}
\begin{remark}
By minor modifications, the estimate for $\|u(t,\cdot)\|_{L^\infty}$ may be improved in the proof of Theorem~\ref{thm:inf} and, consequently, in its statement, but we omitted this improvement for the sake of brevity.
\end{remark}

\section{Conclusion}\label{sec:conclusion}

Theorem~\ref{thm:main} has the advantage of providing an efficient result assuming small initial data only in $L^1$ (the minimal assumption to apply $L^1-L^q$ estimates to treat the nonlinear problem). The main obstacle, however, is that the range of space dimension in which the result holds is bounded.

On the other hand, Theorem~\ref{thm:inf} has the advantage of providing a broad application of the existence result in any space dimension, by only assuming small initial data in $L^1\cap L^p$ (notice that $p\nearrow \infty$ as $n\nearrow\infty$).

However, the both results suffer of the fact that $q_c\geq \gamma(n-1)$ for sufficiently large space dimension~$n$ (Remark~\ref{rem:Strauss}). This may suggest that our approach, based on the $L^1-L^q$ estimates developed in~\cite[Theorem 6]{D'Abbicco-2025}, is of interest only in low space dimension.

On the other hand, the fact that it provides sharp result, in view of the counterpart results of blow-up in finite time for subcritical and critical exponents (see Examples~\ref{ex:2} and~\ref{ex:3}), means that the methodology is perfectly suitable to solve these problems.

It remains open the problem to prove more general blow-up results which may show the sharpness of the critical exponent~$q_c$. Or, otherwise, to find a different approach that allows to lower the existence exponent below the value $q_c$ in this paper (see also~\textsection\ref{sec:qp}).



\section*{Funding}

The author is supported by PRIN 2022 ``Anomalies in partial differential equations and applications'' CUP H53C24000820006 from Ministero dell'Universit\`a e Ricerca (Italia) and he is member of INdAM-GNAMPA group.

%
%
%
%
%
%
%
%
%
%
%
%

\begin{appendices}

\section{The critical exponent using $L^p-L^q$ estimates with $p>1$}\label{sec:qp}

In view of the competition between dissipation and dispersion, one may wonder if the decay rate of $L^1-L^q$ estimate, with $q=1+\alpha$, may be worse than the decay of the $L^p-L^q$ estimate, with $q=p(1+\alpha)$, for some $p\in(1,2)$. This phenomenon is reasonable in view of the oscillations, in particular, see~\eqref{eq:deltapq} with $\theta_0\nearrow\infty$, that is, for models closer to the semilinear wave equation without damping. Indeed, we have the following result, analogous to Theorem~\ref{thm:DAE}.
\begin{theorem}\label{thm:DAEp}[Theorem 6 in~\cite{D'Abbicco-2025}]
Assume $\theta_0>1$. Fix $p\in(1,2)$ and $q\in[p,p']$, where $p'=p/(p-1)$, such that the quantity
\[ d(p,q)=\frac{n}p-\frac{n-1}2 -\frac1q, \]
verifies $d(p,q)>1$. Then the solution to~\eqref{eq:CPgenlin} verifies the following low-frequencies $L^1-L^q$ estimate
\begin{equation}\label{eq:deltapq}\begin{split}
&\|P_0u(t,\cdot)\|_{L^q} \leq C\,(1+t)^{\delta_q}\,\|u_1\|_{L^p}, \qquad t\geq0,\\
&\delta_{p,q}=(n-1)\left(\frac1q-\frac12\right) - \frac1{\theta_0}\left(\frac{n}p-\frac{n+1}2 -\frac1q\right),
\end{split}\end{equation}
for some $C>0$.
\end{theorem}
For a given $\alpha>0$, we are interested in studying the decay rate $\delta_{p,q}$ with $q=p(1+\alpha)$. Taking the derivative with respect to $1/p$ in
\[ \delta_{p,p(1+\alpha)} = (n-1)\left(\frac1{p(1+\alpha)}-\frac12\right) - \frac1{\theta_0}\left(\frac{n}p-\frac{n+1}2 -\frac1{p(1+\alpha)}\right), \]
we find
\[ \partial_{1/p} \delta_{p,p(1+\alpha)} = \frac{n-1}{1+\alpha} - \frac1{\theta_0}\left(n-\frac1{1+\alpha}\right)= \frac{n-1+1/\theta_0}{1+\alpha} - \frac{n}{\theta_0}. \]
Due to the constant sign, this means that either $L^1-L^{1+\alpha}$ is the best estimate among $L^p-L^q$ estimates, to prove the existence argument, as used in this paper, or $L^{1+\frac1{1+\alpha}}-L^{2+\alpha}$ is the best one. The change in the best strategy happens when the critical exponent $q_c$ obtained using $L^1-L^{1+\alpha}$ become smaller than
\[ \frac{(n-1)\theta_0+1}{n} = \theta_0 - \frac{\theta_0-1}n. \]
Replacing the value for $q_c$, we find
\[\begin{split}
& q_c = \frac2{\theta_0+1}\,\frac{n\theta_0+1}{n-1}<\theta_0 - \frac{\theta_0-1}n \\
& \qquad \iff n> \tilde n(\theta_0)=\frac{2\theta_0^2+\theta_0+1+\sqrt{8\theta_0^3+9\theta_0^2-2\theta_0+1}}{2\theta_0(\theta_0-1)}.
\end{split}  \]
For $\theta_0=2$, this corresponds to
\[ n> \tilde n(2) = \frac{11+\sqrt{97}}4, \qquad \text{that is,} \quad n\geq5. \]
Comparing $\tilde n(\theta_0)$ with $n(\theta_0)$ in Remark~\ref{rem:Strauss}, unfortunately, we find
\[ \begin{split}
& n(\theta_0)=\frac{\theta_0^2+\theta_0+2}{\theta_0(\theta_0-1)} < \tilde n(\theta_0) \\
& \qquad \iff \theta_0+1<\sqrt{8\theta_0^3+9\theta_0^2-2\theta_0+1} \iff 2\theta_0^2+2\theta_0-1 >0.
\end{split} \]

\section{The critical exponent for non integrable data}

If the initial data are not assumed to be small in $L^1$, but only in $L^p$, for some $p>1$, it is well-known that the critical exponent is expected to change, in general (see, for instance, \cite{NakaoOno1993,IT05,DA21INDAM,DAbbicco2025JEEQ}). In this case, the existence exponent may be computed solving the equation $(1+\alpha)\delta_{p,p(1+\alpha)}=-1$ (provided that $p\leq 1+1/(1+\alpha)$), and the critical case belongs to the existence range.

However, in the case $p>1$ it may happen that $d(p,q)\leq1$, so that Theorem~\ref{thm:DAEp} is no longer applicable. In this case, the decay rate is expected to only depend on the dispersive action, and to be independent of the damping. Namely, we expect to find the same decay rate for the wave equation:
\[ \|P_0u(t,\cdot)\|_{L^q} \leq C\,(1+t)^{1-n\left(\frac1p-\frac1q\right)}\,\|u_1\|_{L^p}, \qquad d(p,q)\leq1. \]
Replacing $q=(1+\alpha)p$, we find
\[ -1 = (1+\alpha) \left(1-\frac{n}p\left(1-\frac1{1+\alpha}\right)\right) \iff \left(\frac{n}p-1\right)\alpha =2, \]
that is, the critical exponent is the Fujita (or Kato) exponent:
\[ q_c(p)=1 + \frac{2}{\frac{n}p-1} = \frac{n+p}{n-p}\,. \]
The result is valid provided that $pq_c\leq p'$, that is,
\[ p-1\leq \frac1{q_c} = \frac{n-p}{n+p} \iff np+p^2-n-p \leq n-p \iff n\geq n_0(p)=\frac{p^2}{2-p}\,, \]
and that $d(p,pq_c)\leq1$, that is,
\[ \begin{split}
\frac{n}p-\frac{n-1}2 -\frac1{pq_c} \leq1 & \iff \frac{n}p - \frac{n-p}{p(n+p)} \leq \frac{n+1}2 \\
    & \iff 2n^2+2np-2n+2p\leq n^2p+np^2+np+p^2 \\
    & \iff (2-p)n^2 +(p-2-p^2)n +2p-p^2 \leq0\\
    & \iff n\leq n_1(p)=\frac{p^2-p+2+\sqrt{p^4-6p^3+21p^2-20p+4}}{2(2-p)}.
\end{split} \]
We notice that $n_0(p)<n_1(p)$ for any $p\in(1,2)$, and both diverge at $\infty$ as $p\nearrow2$.

%
%
%
%
\end{appendices}


\bibliography{ref}


\end{document}